%% file: wqs.tex
\documentclass[envcountsame,envcountsect]{svmult}

\usepackage{amsmath}
\usepackage{amssymb}
\usepackage{latexsym}
\usepackage{amsxtra}
\usepackage{enumerate}
\usepackage[all]{xy}

\usepackage{mathptmx}       
\usepackage{helvet}         
\usepackage{courier}        
\usepackage{type1cm}        
\usepackage{makeidx}         
\usepackage{graphicx}        
\usepackage{multicol}        
\usepackage[bottom]{footmisc}
\usepackage{url}

\makeindex             

\usepackage[utf8]{inputenc}
\usepackage{longtable}
\usepackage{graphicx}

\spnewtheorem{assumption}[theorem]{Assumption}{\bf}{\rm}

\newcommand{\ZZ}{\mathbb{Z}}
\newcommand{\QQ}{\mathbb{Q}}
\newcommand{\FF}{\mathbb{F}}

\newcommand{\CC}{\mathbb{C}}
\newcommand{\PP}{\mathbb{P}}

\newcommand{\Gal}{\mathop{\rm Gal}\nolimits}
\newcommand{\m}{\mathfrak{m}}
\newcommand{\p}{\mathfrak{p}}
\newcommand{\OO}{\mathcal{O}}
\newcommand{\Spec}{\mathop{\rm Spec}}
\newcommand{\Frac}{\mathop{\rm Frac}}

\newcommand{\Xb}{\bar{X}}

\newcommand{\sing}{{\rm sing}}

\begin{document}

\title*{Desingularization of arithmetic surfaces: algorithmic aspects}

\author{Anne Fr\"uhbis-Kr\"uger and Stefan Wewers}
\authorrunning{Fr\"uhbis-Kr\"uger, Wewers}
\institute{Institut f\"ur Algebraische Geometrie, Leibniz Universität Hannover\\
\email{anne@math.uni-hannover.de}\\
Institut f\"ur Reine Mathematik, Universit\"at Ulm\\
\email{stefan.wewers@uni-ulm.de}}

\maketitle

\abstract{The quest for regular models of arithmetic surfaces allows different viewpoints and approaches: using valuations or a covering by charts. In this article, we sketch both approaches and then show in a concrete example, how 
surprisingly beneficial it can be to exploit properties and techniques from 
both worlds simultaneously.}

\section{Introduction}
\label{sec:intro}

Resolution of singularities in dimension $2$ was first proved by Jung in 1908
\cite{Jung08}, but it was not until Hironaka's work in 1964 \cite{Hironaka64}
that this could also be mastered in dimensions beyond $3$. However, Hironaka's
result only applies to characteristic zero, but not to positive or mixed
characteristic. There the general question is still wide open with partial
results for low dimensions. In particular, Lipman gave a construction for
$2$-dimensional schemes in full generality in \cite{Lipman78}.

Lipman's result includes the case of an {\em arithmetic surface},
i.e. integral models of curves over number fields. In fact, the existence of
(minimal) regular models of curves over number fields is a cornerstone of
modern arithmetic geometry. Important early results are for instance the
existence of a minimal regular model of an elliptic curve by N\'eron
(\cite{Neron64}) and {\em Tate's algorithm} (\cite{Tate75}) for computing it
explicitly. 

In this paper we study a particular series of examples of surface singularites
which is a special case of a construction due to Lorenzini
(\cite{Lorenzini10}, \cite{Lorenzini14}). The singularity in question is a
{\em wild quotient singularity}. More precisely, the singular point lies on an
arithmetic surface of mixed characteristic $(0,p)$ which is the quotient of a
regular surface by a cyclic group of prime order $p$, such that the group
action has isolated fixed points. We prove that in our example one obtains a
series of rational determinantal singularities of multiplicity $p$, and we are
able to write down explicit equations for these (see Proposition
\ref{prop:example1}).

Determinantal rings (of expected codimension) are well-studied objects in 
commutative algebra: the free resolution is the Eagon-Northcott complex and 
hence many invariants of the ring such as projective dimension, depth, 
Castelnuovo-Mumford regularity, etc. are known (see e.g. \cite{EisenbudBook},
\cite{BrunsVetter}). Beyond that, such singularities (in the geometric case)
are an active area of current research in singularity theory studying e.g. 
classification questions, invariants, notions of equivalence and topological 
properties, see e.g. \cite{FN}, \cite{NOT}, \cite{Z}. We show, by a direct 
computation, that the resolution in our arithmetic setting is completely 
analogous to the geometric case.

Both for deriving the equations of our singularities and for resolving them,
we employ and mix two rather different approaches to represent and to compute
with arithmetic surfaces. The first approach is more standard and consists in
representing a surfaces as a finite union of affine charts, and the coordinate
ring of each affine chart as a finitely generated algebra over the ground
ring. From this point of view, computations with arithmetic surfaces can be
performed with standard tools from computer algebra, like standard
bases (e.g. in {\sc Singular} \cite{DGPS}). However, these techniques are not 
yet as mature in the arithmetic case as they are in the geometric case.

The second approach uses valuations as its main tool. We work over a discrete
valuation ring $R$. An arithmetic surface $X$ over $\Spec R$ is considered as
an $R$-model of its generic fiber $X_K$ (a smooth curve over $K={\rm
  Frac}(R)$). Then any (normal) $R$-model $X$ of $X_K$ is determined by a
finite set $V(X)$ of discrete valuations on the function field of $X_K$
corresponding to the irreducible components of the special fiber of $X$. A
priori, it is not clear how to extract useful information about the model $X$
from the set $V(X)$. Nevertheless, in joint work with J.\ R\"uth the second
named author has used this technique successfully for computing semistable
reduction of curves (see e.g.\ \cite{superp}).

\vspace{2ex} The paper is structured as follows. In Section
\ref{sec:background} we give some general definitions concerning arithmetic
surfaces, and we present our two approaches for representing them explicitly.
Section \ref{sec:wildquot} then presents our series of wild quotient
singularities. In the final section, we compute, in one concrete example of
our wild quotient singularities, an explicit desingularization.

\section{Arithmetic surfaces and models of curves}\label{sec:background}

\subsection{General definitions}

\begin{definition}
  By a {\em surface} we mean an integral and noetherian scheme $X$ of
  dimension $2$. An {\em arithmetic surface} is a surface $X$ together with a
  faithfully flat morphism $f:X\to S=\Spec(R)$ of finite type, where $R$ is a
  Dedekind domain. To avoid technicalities, we always assume that $R$ (and
  hence $X$) is excellent. Moreover, we will assume in addition that $X$ is
  normal, unless we explictly say otherwise.
\end{definition}

A common situation where arithmetic surfaces occur is the following. Let $R$
be a Dedekind domain, $K=\Frac(R)$ and $X_K$ a smooth and projective curve
over $K$. An {\em $R$-model} of $X_K$ is an arithmetic surface $X\to\Spec(R)$,
together with an identification of $X_K$ with the generic fiber of $X$,
i.e.\ $X_K=X\otimes_R K$.

\vspace{1ex}
For the following discussion we fix an arithmetic surface
$X\to\Spec(R)$. We write $X^\sing$ for the subset of points whose
local ring is not regular. Since we assume that $X$ is normal, $X^\sing$ is
closed of codimension $2$ and hence consists of a finite set of closed points
of $X$. A point $\xi\in X^\sing$ is called a {\em
  singularity} of $X$. (If we drop the normality condition, then $X^\sing$ may
also have components of codimension $1$.)

By a {\em modification} of $X$ we mean a proper
birational map $f:X'\to X$. A modification is an isomorphism outside a finite
set of closed points. If  $f$ is an isomorphism away
from a single point $\xi\in X$, then $\xi$ is called the {\em center} of the
modification and $E:=f^{-1}(\xi)\subset X'$ the {\em exceptional fiber} or {\em 
exceptional locus} (we endow $E$ with the reduced subscheme structure). Note
that $E$ is a connected scheme of dimension one. We will use the notation 
\[
   E = \cup_{i=1}^n C_i,
\]
where the $C_i$ are the irreducible components. Each of them is a projective
curve over the residue field $k=k(\xi)$. If the modification changes more 
than a single point, we will still denote the exceptional locus by $E$, but
$E$ obviously does not need to be connected any more.

\begin{definition}
  Let $p:X\to S$ be an arithmetic surface and $\xi\in X^\sing$ a
  singularity. A {\em desingularization} of $\xi\in X$ is a modification
  $f:X'\to X$ with center $\xi$ and exceptional fiber $E=f^{-1}(\xi)$ such
  that every point $\xi'\in E$ is a regular point of $X'$. A desingularization
  of $X$ is a modification consisting of desingularizations at all points of
  $X^\sing$.
\end{definition}

By a theorem Lipman (\cite{Lipman78}), a desingularization of $X$ always
exists by means of a sequence of normalizations and blow-ups.  Depending on
the situation we often want $f$ to satisfy further conditions.  We list some
of them:
\begin{enumerate}[(a)]
\item
  The exceptional divisor $E$ is a normal crossing divisor of $X'$.
\item
  Let $s:=p(x)$. Then the fiber $X_s'$ of $X'$ over $s$ is a normal crossing
  divisor on $X'$ (when endowed with the reduced subscheme structure). 
\item
  The desingularization $f:X'\to X$ is minimal (among all desingularizations
  of $\xi\in X$).
\item
  $f:X'\to X$ is minimal among all desingularizations satisfying (a) (resp.\
  (b)).
\end{enumerate}

Choosing a different approach than Lipman and avoiding normalizations 
completely, Cossart, Janssen and Saito proved a desingularization algorithm
relying only on blow-ups at regular centers in \cite{CJS}, see also 
\cite{CSch}. The approach allows to additionally satisfy yet another rather 
common condition:
\begin{enumerate}
\item[(e)] If $X \subset W$ for some regular scheme\footnote{as before $W$ 
should be excellent, noetherian, integral}, then desingularization of $X$ can 
be achieved by modifications of $W$ which are isomorphisms outside $X^\sing$. 
\end{enumerate} 

\subsection{Presentation by affine charts}\label{subsec:charts}

We are interested in the problem of computing a desingularization $f:X'\to X$
of a given singularity $\xi\in X$ on an arithmetic surface explicitly. Before
we can even state this problem precisely, we have to say something about the
way in which the surface $X$ is represented.  

The most obvious way\footnote{thanks to Grothendieck} to present $X$ is to
write it as a union of affine charts,
\[
    X = \cup_{j=1}^r U_j, \quad U_j = \Spec A_j.
\]
Here each $A_j$ is a finitely generated $R$-algebra whose fraction field is
the function field $F(X)$ of $X$. After choosing a set of generators of
$A_j/R$, we can obtain a presentation `by generators and relations'. This
means that 
\[
    A_j = R[\underline{x}]/I_j,
\]
where $\underline{x}=(x_1,\ldots,x_{n_j})$ is a set of indeterminates and
$I_j\lhd R[\underline{x}]$ is an ideal. Choosing a list of generators of
$I_j$, we obtain a presentation
\[
    R[\underline{x}]^{m_j} \to R[\underline{x}]\to A_j \to 0.
\]
Taking into account the relations among the generators of the ideal $I_j$ this
presentation extends to 
\[
    R[\underline{x}]^{n_j} \to R[\underline{x}]^{m_j} \to R[\underline{x}]\to A_j \to 0,
\]
where the matrix describing the left-most map is usually referred to as the
first syzygy matrix of $I_j$ or $A_j$ respectively. Iteratively forming 
higher syzygies, this leads to free resolutions, i.e. exact sequences of free
$R[\underline{x}]$-modules. As $R[\underline{x}]$ is a polynomial ring over
a Dedekind domain, it has global dimension $n_j+1$ and hence $A_j$ possesses
a free resolution of length at most $n_j +1$. 
Working locally at a maximal ideal ${\mathfrak m} \subset R[\underline{x}]$, 
this allows e.g. the calculation of the ${\mathfrak m}$-depth of $A_j$ by 
the Auslander-Buchsbaum formula. 

In the subsequent sections, we shall encounter examples placing us in a 
particular situation, for which free resolutions are well understood: 
determinantal varieties corresponding to maximal minors. For these, 
$I_j$ is generated by the maximal minors of an $m \times n$ matrix defining 
a variety of codimension $(m-t+1)(n-t+1)$, where 
$t=\operatorname{min}\{m,n\}$. Most prominently, the Hilbert-Burch theorem 
(see for instance \cite{EisenbudBook}) relates Cohen-Macaulay codimension $2$ 
varieties to the $t$-minors of their first syzygy matrix, which is of size 
$t \times (t+1)$, and ensures the map given by this matrix to be injective.

\subsection{Presentation using valuations} \label{subsec:valuations}

An alternative way\footnote{Historically, this was actually the first method,
  pioneered by Deuring \cite{Deuring42} more than 10 years before the
  invention of schemes.} to present an arithmetic surface is the following. To
describe it it is convenient to assume that $R$ is a local ring. Then $R$ is
actually the valuation ring of a discrete valuation $v_K:K^\times\to\QQ$ of
its fraction field $K=\Frac(R)$. We choose a uniformizer $\pi$ of $v_K$ (i.e.\
a generator of the maximal ideal $\p\lhd R$) and normalize $v_K$ such that
$v_K(\pi)=1$. We denote the residue field of $v_K$ by $k$. In addition we make
the following assumption\footnote{More generally, we could have assumed that
  $(K,v_K)$ satisfies the {\em local Skolem property}, see
  \cite{GreenPopMatignon95}}:

\begin{assumption} \label{ass:LocalSkolem}
  The valuation $v_K$ is either henselian, or its residue field $k$ is
  algebraic over a finite field.
\end{assumption}

We fix a smooth projective curve $X_K$ over $K$. Note that $X_K$ is uniquely
determined by its function field $F_X$, and conversely every finitely
generated field extension $F/K$ of transcendence degree $1$ is the function
field of a smooth projective curve $X_K$. 

Let $X$ be an $R$-model of $X_K$, $X_s$ its special fiber and 
\[
       X_s = \cup_i \Xb_i
\]
its decomposition into irreducible components. Then each component $\Xb_i$ is
a prime divisor on the surface $X$. Because $X$ is normal, $\Xb_i$ gives rise
to a discrete valuation $v_i$ on $F_X$ such that $v_i(\pi)>0$. We normalize
$v_i$ such that $v_i(\pi)=1$. i.e.\ such that $v_i|_K=v_K$. By definition, the
residue field $k(v_i)$ of $v_i$ is the function field of the component
$\Xb_i$. In particular, $k(v_i)$ is function field over $k$ of transcendence
degree $1$. 

A discrete valuation $v$ on the function field $F_x$ is called {\em geometric}
if $v|_K=v_K$ and the residue field $k(v)$ is a finitely generated extension
of $k$ of transcendence degree $1$. Let $V(F_X)$ denote the set of geometric
valuations. Given a model $X$ of $X_K$, we write
\[
    V(X) := \{v_1,\ldots,v_r\}\subset V(F_X)
\]
for the set of geometric valuations corresponding to the components of the
special fiber of $X$. 

\begin{theorem}  \label{thm:valuations}
  The map
  \[
       X \mapsto V(X)
  \]
  is a bijection between the set of isomorphism classes of $R$-models of $X_K$
  and the set of finite nonempty subsets of $V(F_X)$. 
  
  Furthermore, given two models $X,X'$ of $X_K$, there exists a map $X'\to X$
  which is the identity on $X_K$ (and which is then automatically a
  modification) if and only if $V(X)\subset V(X')$.
\end{theorem}

\begin{proof}
See \cite{Green96} or \cite{JulianDiss}.
\hspace*{\fill}\qed
\end{proof}

By the above theorem models of a given smooth projective curve $X_K$ over a
valued field $(K,v_K)$ can be defined simply by specifying a finite list of
valuations. An obvious drawback of this approach is that it is not obvious how
to extract detailed information on the model $X$ from the set $V(X)$. A
priori, $V(X)$ only gives `birational' information on the special fiber
$X_s$. For instance, it is not immediate to see whether the model $X$ is
regular. 

So far, the above approach based on valuations has proved to be very useful
for the computation of semistable models (see \cite{superp}). We intend to
extend it to other problems in the future. In \S \ref{subsec:aposteriori} we
will see a first attempt to use it for desingularization.

\subsection{Computational tools}

In this section we report on some ongoing work to implement computational
tools for dealing with arithmetic surfaces and their desingularization.

\subsubsection*{Valuation based approach}
As we have explained in \S \ref{subsec:valuations}, it is in principle
possible to describe arithmetic surfaces over a local field purely in terms of
valuations. In order to use this approach for explicit computations, one
needs a way to write down, manipulate and compute with geometric
valuations. Fortunately, such methods are available (but maybe not as widely
known as they should). Our approach goes back to work of MacLane
(\cite{MacLane36a}, \cite{MacLane36b}). In the present context
(i.e.\ for describing models of curves over local fields) it has been
developed systematically in Julian R\"uth's PhD thesis (\cite{JulianDiss}). 

We will not go into details, but for later use we need to introduce the notion
of an {\em inductive valuation}. Let $K$ be a field with a discrete valuation
$v_K$ and valuation ring $R$ as before. Let $v$ be an extension of $v_K$ to a
geometric valuation on the rational function field $K(x)$. We assume in
addition that $v(x)\geq 0$ (i.e.\ that $R[x]$ is contained in the valuation
ring of $v$). Let $\phi\in R[x]$ be a monic integral polynomial, and let
$\lambda\in\QQ$ be a rational number satifying $\lambda>v(\phi)$. If $\phi$ is
a {\em key polynomial} for $v$ (see \cite{JulianDiss}, Definition 4.7) then we
can define a new geometric valuation $v'$ (called an {\em augmentation} of
$v$) with the property that
\[
   v'(\phi) = \lambda,\quad 
   v'(f) = v(f)\;\;\text{for $f\in K[x]$ with  $\deg(f)<\deg(\phi)$.}
\] 
See \cite{JulianDiss}, Definition 4.9. We write 
\[
    v' = [v,\;v'(\phi)=\lambda].
\]

The process of augmenting a given geometric valuation can be iterated. A
geometric valuation $v$ on $K(x)$ which is obtained by a sequence of
augmentations, starting from the Gauss valuation with respect to $x$, is
called an {\em inductive valuation}. It can be written as 
\begin{equation} \label{eq:valuations3}
    v = v_n = [v_0,v_1(\phi_1)=\lambda_1,\ldots, v_n(\phi_n)=\lambda_n].
\end{equation}
Here $v_0$ is the Gauss valuation, $\lambda_i\in\QQ$ and $\phi_i\in R[x]$ is
monic. Furthermore, $\phi_i$ is a key polynomial for $v_{i-1}$ and
$\lambda_i>v_{i-1}(\phi_i)$. By \cite{JulianDiss}, Theorem 4.31, every
geometric valuation $v$ on $K(x)$ with $v(x)\geq 0$ can be written as an
inductive valuation. 

The notion of inductive valuation can be extended in several
ways. Firstly, by replacing $x$ with $x^{-1}$ if necessary, we can drop the
condition $v(x)\geq 0$, Hence we can write every geometric valuation on $K(x)$
as an inductive valuation. Secondly, for the last augmentation step in
\eqref{eq:valuations3} we can allow the value $\lambda_n=\infty$. The
resulting $v_n$ is then only a {\em pseudo-valuation} and induces a true
valuation on the quotient ring $L:=K[x]/(\phi_n)$ (which is a field because
key polynomials are irreducible). Thirdly, given an arbitrary finite extension
$L/K$, we can compute the (finite) set of extensions $w$ of $v_K$ to $L$ as
follows. We write $L=K[x]/(f)$ for an irreducible polynomial $f\in K[x]$. If
$f$ is irreducible over the completion $\hat{K}$ of $K$ with respect to $v_K$,
then there exists a unique extension $w$ of $v$ to $L$ which can be written as
an inductive pseudo-valuation on $K[x]$ (with $\phi_n=f$). In general, let
$f=\prod_i f_i$ be the factorization into irreducibles over $\hat{K}$. Then
each factor $f_i$ gives rise to an extension $w_i$ of $v$ to $L$. Considering
$w_i$ as a pseudo-valuation on $K[x]$, MacLane shows that $w_i$ can be written
as a {\em limit valuation} of a chain of inductive valuations $v_n$. By this
we mean that $v_n$ is an augmentation of $v_{n-1}$, and for every
$\alpha=(g(x)\mod {(f)}) \in L$ there exists $n\geq 0$ such that
$w_i(\alpha)=v_n(g)=v_{n+1}(g)=\ldots$. 

MacLane's theory is constructive and can be used to implement algorithms for
dealing with discrete valuations on a fairly large class of fields. 
A Sage package written by Julian R\"uth called {\tt mac\_lane} 
(\cite{RuethMacLane})
is availabel under \url{github.com/saraedum/mac_lane}. It can be use to define
and compute with discrete valuations of the following kind:
\begin{itemize}
\item
  $\p$-adic valuations on number fields.
\item
  Geometric valuations $v$ on function fields $F/K$ (of dimension $1$) whose
  restriction to $K$ is either trivial, or can be defined by this package.
\end{itemize}
Given a valuation $v$ on a field $K$ of the above kind and a finite separable
extension $L/K$, it is possible to compute the set of all extension of $v$ to
$K$.

\subsubsection*{Chart based approach}

On the other hand, a description by affine charts as in \ref{subsec:charts}
not only emphasizes the similarity to the geometric setting, it also allows 
the use of computational techniques such as standard bases (whenever a 
suitably powerful arithmetic for computations in $R$ is available). This, 
in turn, opens up a whole portfolio of algorithms ranging from basic 
functionality like elimination or ideal quotients to more sophisticated 
algorithms such as blowing up and normalization, which eventually
permit to practically implement the above mentioned algorithms of Lipman and
of Cossart-Janssen-Saito for desingularization of $2$-dimensional schemes. 
Note at this point that neither of the two algorithms imposes the condition
of normality on the surfaces to be resolved.\\

In a nutshell, the desingularization problem for $2$-dimensional schemes is
the problem of finding suitable centers which improve the singularity without
introducing new complications. In this context, $0$-dimensional centers for
blow-ups usually do not pose any major problems: such blow-ups at different
centers may be interchanged, as they are isomorphisms outside their respective
centers and hence do not interact.  However, even resolving a $0$-dimensional
singular point in the geometric case may already require the use of
$1$-dimensional centers to achieve a regular model and normal crossing
divisors. These curves can exhibit significantly more structure than sets of
points, e.g. they can possess intersecting components or non-regular
branches. So the central problems in
resolving the singularities of $2$-dimensional schemes are ensuring
improvement in each step and treating $1$-dimensional loci which need to be
improved. In particular for the latter, the two aforementioned approaches
differ significantly. \\

The key idea behind Lipman's algorithm \cite{Lipman78} is that normal
varieties are regular in codimension $1$, i.e. that their singular locus is
$0$-dimensional. Thus a normalization step can always ensure that only sets of
points will be required for subsequent blowing up:

\begin{theorem}[\cite{Lipman78}]
Let $X$ be an excellent, noetherian, reduced scheme of dimension $2$, then $X$
posses a desingularization by a finite sequence of birational morphisms of
the form
$$X_r \stackrel{\pi_r \circ n_r}{\longrightarrow} \dots 
      \stackrel{\pi_2 \circ n_2}{\longrightarrow} X_1 
      \stackrel{\pi_1 \circ n_1}{\longrightarrow} X_0=X,$$
where $\pi_i$ denotes a blow up at a finite number of points, $n_i$ a 
normalization and $X_r$ is regular.
\end{theorem}

While blowing up is algorithmically straightforward e.g. using an elimination 
(see e.g. \cite{FK1}), the hard step is the normalization. Although there has
been significant improvement in the efficiency of Grauert-Remmert style 
normalization algorithms in the last decade (see e.g. \cite{GLSe}, 
\cite{BDetal}), this is still a bottleneck when working over a Dedekind 
domain $R$ instead of a field. The crucial step here is the choice of a 
suitable test ideal, i.e. a radical ideal contained in the ideal of the 
non-normal locus and containing a non-zerodivisor. In the geometric case, 
the ideal of the singular locus -- generated by the original set of generators
and the appropriate minors of the Jacobian matrix -- is well-suited for this
task, but in the current setting it also sees fibre singularities which do 
not contribute to the non-regular locus. Hence the approximation of the 
non-normal locus by this test ideal is rather coarse and significantly 
impedes efficiency. In practice, a better approximation of the non-normal locus
is achieved by constructing a test ideal following an idea of Hironaka's
termination criterion: we use the locus where Hironaka's invariant $\nu^*$, 
i.e. the tuple of orders (in the sense of orders of power series) of the 
elements of a local standard basis, sorted by increasing order, is 
lexicographically greater than a tuple of ones.\\
 
The approach of Cossart-Janssen-Saito \cite{CJS} (CJS for short) on the other 
hand, avoids normalization completely and allows well-chosen $1$-dimensional 
centers, whenever necessary; when choosing centers, it takes into account 
the full history of blowing ups leading to the current situation. In 
constrast to Lipman's approach, this algorithm yields an embedded 
desingularization. Nevertheless, a key step is again the use of the locus
where $\nu^*$ lexicographically exceeds a tuple of ones. But then, no 
normalization follows, instead the singularities of this locus are first 
resolved before it is itself used as a $1$-dimensional center. Each arising
exceptional curve in this process remembers when it was created and whether
its center was of dimension $0$ or $1$, because this information is crucial
in the choice of center for ensuring improvement as well as normal crossing
of exceptional curves.

A beta version of the first algorithm is available as {\sc Singular}-library
reslipman.lib and is planned to become part of the distribution in the near 
future. A prototype implementation of the CJS-algorithm has been implemented 
and is closely related to an ongoing PhD-project on a parallel approach to 
resolution of singularities using the gpi-space parallelization environment
(for recent progress along this train of thought see \cite{BF}, \cite{gpi}).

\section{Explicit construction of wild quotient singularities} \label{sec:wildquot}

In this section we describe a series of examples for arithmetic surfaces with
interes\-ting singularities. The general construction is due to Lorenzini (see
\cite{Lorenzini10} and \cite{Lorenzini14}). Our contribution is to 
explictly describe the (local) ring of the singularity by generators and
relations. In the next section we also describe the desingularization in an
equally explicit way. 

Let $R$ be a discrete valuation ring, with maximal ideal $\p$, residue field
$k=R/\p$ and fraction field $K$. Let $v_K$ denote the corresponding discrete
valuation on $K$. We assume that $k$ has positive characteristic $p$ and that
$v_K$ is henselian (in particular, Assumption \ref{ass:LocalSkolem} holds). 

Let $X_K$ be a smooth, projective and absolutely irreducible curve over $K$,
of genus $g$. We assume that $X_K$ has potentially good reduction
reduction with respect to $v_K$. This means that there exists a finite
extension $L/K$ and a smooth model $Y$ of $X_L:=X_K\otimes_K L$ over the
integral closure $R_L$ of $R$ in $L$. Note that $R_L$ is a discrete valuation
ring corresponding to the unique extension $v_L$ of $v_K$ to $L$. We assume in
addition that $L/K$ is a Galois extension, and that the natural action of
$G:=\Gal(L/K)$ on $X_L$ extends to an action on $Y$. Under this assumption, we
can form the quotient scheme $X_Y/G$. It is an $R$-model of $X_K$. 

The model $Y$ is regular because $Y\to\Spec(R)$ is smooth by assumption.
However, the quotient scheme $X=Y/G$ may have singularities. In fact, let
$\xi\in X_s$ be a closed point on the special fiber of $X$, and let $\eta\in
Y_s$ be a point above $\xi$. Let $I_\eta\subset G$ denote the inertia subgroup
of $\eta$ in $G$. If $I_\eta=1$ then the map $Y\to X$ is \'etale in $\eta$. It
follows that $X$ is regular in $\xi$ because $Y$ is regular in $\eta$.

In general, the locus of points with $I_\eta\neq 1$ may consists of the entire
closed fiber $Y_s$ and hence be a subset of codimension $1$ on $Y$. To obtain
isolated quotient singularities we impose the following condition:

\begin{assumption} \label{ass:genericallyetale}
  The action of $G$ on the special fiber $Y_s$ is generically free.
\end{assumption}

Under this assumption, there are at most a
finite number of points $\eta\in Y_s$ with nontrivial
inertia $I_\eta\neq 1$. Let $\xi_1,\ldots,\xi_r\in X_s$ be the images of the
points $\eta\in Y_s$ with $I_\eta\neq 1$. Then $\xi_1,\ldots,\xi_r$ are
precisely the singularities of the model $X$.

\begin{remark}
  In Lorenzini's original setting, Assumption \ref{ass:genericallyetale} holds
  automatically because the curve $Y$ has genus $g(Y)\geq 2$. In our series of
  examples we have $g(Y)=0$, but the assumption holds nevertheless.
\end{remark} 

\subsection{An explicit example}

Let $p$ be a prime number, $K$ a number field and $\p\mid p$ a prime ideal of
$\OO_K$ over $p$. Let $v_K$ denote the discrete valuation on $K$ corresponding
to $\p$ and $R$ the valuation ring of $v_K$. Let $L/K$ be a Galois extension
of degree $p$ which is totally ramified at $\p$. This means that $v_K$ has a
unique extension $v_L$ to $L$. Let $\sigma$ be a generator of the cyclic group
$G=\Gal(L/K)$. Let $\pi_L$ be a uniformizer for $v_L$. We normalize $v_L$ such
that $v_L(\pi_L)=1$. Set
\[
    m := v_L(\sigma(\pi_L)-\pi_L).
\]
Then $m\geq 2$ is the first and only break in the filtraton of $G$ by higher
ramification groups. We let $u\in k^\times$ denote the image of the element
$(\sigma(\pi_L)-\pi_l)/\pi_L^m\in R^\times$. 

Let $X_K:=\PP^1_K$ be the projective line over $K$. We identify the function
field $F_X$ with the rational function field $K(x)$ in the indeterminate
$x$. Then $L(x)$ is the function field of $X_L=\PP^1_L$. We define an element
\[
    y := \frac{x-\pi_L}{\pi_L^m} \in L(x).
\]
Clearly, $L(x)=L(y)$, and so $y$, considered as a rational function on $X_L$,
gives rise to an isomorphism $X_L\cong \PP^1_L$. We let $Y$ denote the smooth
$R_L$-model of $X_L$ such that $y$ extends to an isomophism
$Y\cong\PP^1_{R_L}$. By an easy calculation we see that $\sigma(y)=ay+b$, with
$a\in R_L^\times$ and $b\in R_L$. Furthermore,
\[
     \sigma(y) \equiv y + u \pmod{\pi_L}.
\]
In geometric terms this means that the action of $G$ on $X_L$ extends to the
smooth model $Y$, and that the restriction of this action to the special fiber
$Y_s\cong\PP^1_k$ is generically free (and hence Assumption
\ref{ass:genericallyetale} holds). In fact, the action of $G$ is fix point
free on the affine line $\Spec k[y]$, and if $\eta\in Y_s$ denote the point
corresponding to $y=\infty$ then $I_\eta=G$.

Let $\xi\in X_s$ denote the image of $\eta$. By construction, $\xi$ is a wild
quotient singularity, and it is the only singular point on $X$. Our goal is to
write down explictly an affine chart $U=\Spec A\subset X$ containing $\xi$. 

To state our result we need some more notation. Let $\phi\in K[x]$ denote the
minimal polynomial of $\pi_L$ over $K$. Then
\[
    \phi = x^p + \sum_{i=0}^{p-1} a_ix^i = \prod_{k=0}^{p-1} (x-\sigma^k(\pi_L)),
\]
where $a_0,\ldots,a_{p-1}\in\p$. The constant coefficent
\[
     \pi_K:=a_0=N_{L/K}(\pi_L)
\]
is actually a prime element of $R$, i.e. $\phi$ is an Eisenstein
polynomial. 

The following lemma gives a characterization of the model $X$ in
terms of the set $V(X)$ of valuations corresponding to the irreducible
components of the special fiber (as in Theorem \ref{thm:valuations}). 

\begin{lemma} \label{lem:example1}
  We have 
  \[
      V(X) = \{v\}
  \]
  where $v$ is the inductive valuation on $K(x)$ extending $v_K$ given by
  \[
      v := [v(x) = 1/p,\; v(\phi) = m]. 
  \] 
  (See \S \ref{subsec:valuations} and \eqref{eq:valuations3} for
  the relevant notation.)
\end{lemma}

\begin{proof}
It is clear that $V(Y)=\{w\}$, where $w$ is the Gauss valuation on
$F(X_L)=L(y)$ with respect to the parameter $y$ and the valuation $v_L$. Since
$Y\to X=Y/G$ is a finite morphism between (normal) models of their respective
generic fibers, we have $V(X)=\{v\}$, where $v$ is the restriction of $w$ to
the subfield $F(X_K)=K(x)\subset F(X_L)=L(y)$. It remains 
to identify $v$ with the inductive valuation given in the statement of the
lemma. 

We will use the  characterization of an inductive valuation which is implicit
in \cite{JulianDiss}, \S 4.4. Let $v'$ be a valuation on $K(x)$ which
extends $v_K$ and satifies
\[
   v'(x)\geq 0, \quad v'(\phi)\geq m.
\]
Then we claim that $v(f)\leq v'(f)$ for all $f\in K[x]$. By \cite{JulianDiss},
Theorem 4.56, the claim implies that 
\[
    v = [v(x)=1/p,\, v(\phi)=m].
\]
  
To prove the claim, we choose an extension $w'$ of $v'$ to the overfield
$L(y)$. Then 
\begin{equation} \label{eq:exa1.5}
   m \leq v'(\phi) = \sum_{i=0}^{p-1} w'(x-\sigma^i(\pi_L))
     = \sum_{i=0}^{p-1} w'(\pi_L^my+\pi_L-\sigma^i(\pi_L)).
\end{equation}
By definition we have
\begin{equation} \label{eq:exa1.6}
   w'(\pi_L)=v_L(\pi_L)=1/p, \quad 
   w'(\pi_L-\sigma^i(\pi_L))=v_L(\pi_L-\sigma^i(\pi_L))\geq m/p.
\end{equation}
Combining \eqref{eq:exa1.5}, \eqref{eq:exa1.6} and the strong triangle
inequality we conclude that $w'(y)\geq 0$. The valuation $w$ beeing the Gauss
valuation with respect to $y$ and $v_L$ this implies $w(f)\leq w'(f)$ for all
$f\in K[y]$. But $K[x]\subset K[y]$, and therefore $v(f)\leq v'(f)$ for all
$f\in K[x]$. This proves the claim and also the lemma.  \hspace*{\fill}\qed
\end{proof}

Let $D_K\subset X_K$ be the divisor of zeroes of $\phi$, and let $D\subset X$
be the closure of $D_K$. Let $U:=X-D$ denote the complement.

\begin{proposition} \label{prop:example1}
\begin{enumerate}
\item
  We have $U=\Spec A$, where $A\subset F_X=K(x)$ is the sub-$R$-algebra
  generated by the elements $x_0,\ldots,x_{p-1}$, where
  \[
      x_i := \pi_K^mx^i\phi^{-1}, \quad i=0,\ldots,p-1.
  \]
  The point $\xi$ lies on $U$ and corresponds to the maximal ideal
  \[
       \m := (\pi_K,x_0,\ldots,x_{p-1})\lhd A.
  \]
\item
  The ideal of relations between the generators $x_0,\ldots,x_{p-1}$ is
  generated by the $2\times 2$ minors of the matrix
  \[
     M:= \begin{pmatrix}
        x_0     &  x_1  \\
        x_1     &  x_2 \\
        \vdots  &  \vdots \\
        x_{p-2} &  x_{p-1} \\
        x_{p-1} &  z      
     \end{pmatrix}, \quad 
     \text{with $z := \pi_K^m - \sum_{i=0}^{p-1} a_i x_i$.}
  \]
\end{enumerate}
\end{proposition}

\begin{proof} 
  It follows from \cite{Liu}, Corollary 5.3.24, that the divisor $D\subset X$
  is ample, and hence $U:=X-D=\Spec (A)$ is affine. Since $X$ is normal, the
  ring $A$ consists precisely of all rational functions $f\in K(x)$ with ${\rm
    ord}_Z(f)\geq 0$, for any prime divisor $Z\subset X$ distinct from $D$.

  A prime divisor $Z\subset X$ is either horizontal (i.e.\ the closure of a
  closed point on $X_K$) or equal to $X_s$. By Lemma \ref{lem:example1}, $X_s$
  is a prime divisor with corresponding valuation $v$ on $K(x)$. 
  It follows that
  \[
      A = \{ f\in A_K \mid v(f) \geq 0 \},
  \]
  where 
  \[
      A_K = K[\phi^{-1}, x\phi^{-1},\ldots, x^{p-1}\phi^{-1}].
  \]
  In order to make the condition $v(f)\geq 0$ more explicit, we write $f\in
  A_K$ in the form
  \[
      f = c_0 +\sum_{i=0}^{r-1} \sum_{j=0}^{p-1} c_{i,j}x^j\phi^{i-r},
  \]
  with $c_0, c_{i,j}\in K$. 
  Then Lemma \ref{lem:example1} shows that 
  \[
       v(f) = \min\{ v_K(c_0), v_K(c_{i,j})+j/p-m(r-i) \}.
  \]
  So the condition $v(f)\geq 0$ is equivalent to
  \[
        v_K(c_{i,j})+j/p \geq m(r-i),
  \]
  for $i=0,\ldots,r-1$ and $j=0,\ldots,p-1$. It follows that
  \[
     A = R[x_0,\ldots,x_{p-1}], \quad \text{where $x_j:=\pi_K^mx^j\phi^{-1}$.}
  \]
  This is the first part of Statement (i); the second part is obvious.

  To prove Statement (ii) we let $I$ be the ideal in the polynomial ring
  $R[\underline{x}]=R[x_0,\ldots,x_{p-1}]$ generated by the $2\times 2$-minors
  of the matrix $M$. It is easy to check that the generators of $A$ satisfy
  these relations. Therefore, we have a surjective map
  $A':=R[x_0,\ldots,x_{p-1}]/I\to A$. We want to prove that $A'=A$.

  Let $A'':=A'[x_0^{-1}]$ and consider the matrix $M$ with entries in
  $A''$. By definition we have ${\rm rk}M\leq 1$, and the upper left entry
  $x_0$ is a unit. An elementary argument shows that there exists $t\in A''$
  such that 
  \[
     x_0\phi(t)=\pi_K^m, \quad x_i = t^ix_0,\quad i=1,\ldots,p-1.
  \]
  It follows that
  \[
       A''=R[x_0,x_0^{-1},t \mid x_0\phi(t)=\pi_K^m].
  \]
  In particular $A''/R[x_0,x_0^{-1}]$ is a finite flat and generically
  \'etale extension of degree $p$. We deduce that $A''$ is an integral domain
  of dimension $2$. 
  Looking at the equations defining $A'$, it is easy to see that 
  \[
      (x_0)^{\rm rad} = (x_0,\ldots,x_{p-1})
  \]
  and that $A'/(x_0)^{\rm rad}\cong R$ has dimension $1$. Together with $\dim
  A''=2$ this implies that $\dim A'=2$. Therefore, $A'$ is a determinantal
  ring of the `expected' codimension $(p-2+1)(2-2+1)=p-1$. Now a theorem of 
  Eagon and Hoechster shows that $A'$ is Cohen-Macaulay 
  (see \cite{EisenbudBook}, Theorem 18.18 for a textbook reference). Every 
  associated prime of a Cohen-Macaulay ring is minimal
  (\cite{EisenbudBook}, Corollary 18.10). Since $A''=A'[x_0^{-1}]$ is an integral
  domain, it follows that $A'$ is an integral domain as well.

  The analysis of $A''$ from above also shows that 
  \[
          A''_K=A'_K[x_0^{-1}]=A_K[x_0^{-1}] = K[x,\phi^{-1}].
  \]
  It follows that $J={\rm ker}(A'\to A)$ is an ideal of codimension $\geq
  1$. But $A,A'$ have the same dimension, so $J$ consists of zero divisor. On
  the other hand, we have shown above that $A'$ is an integral domain. Hence
  $J=0$. This completes the proof of Proposition \ref{prop:example1}. 
\hspace*{\fill}\qed
\end{proof}

\begin{example} \label{exa:example1}
  The simplest special case of Proposition \ref{prop:example1} where the
  resulting singularity is not a complete intersection is for $p=3$. To make
  this even more explicit, we set $K:=\QQ$ and let $v_K$ denote the $3$-adic
  valuation on $K$ and $R:=\ZZ_{(3)}$ the valuation ring (the localizaton of
  $\ZZ$ at $3$). Moreover, we set
  \[
        \phi := x^3-3x^2+3.
  \]
  The splitting field $L/K$ of $\phi$ is a Galois extension of degree $3$
  which is totally ramified at $p=3$. Indeed, we can factor $\phi$ as
  \[
      \phi = (x-\pi)(x-\sigma(\pi))(x-\sigma^2(\pi)) 
           = (x-\pi)(x-\pi-\pi^2+3\pi)(x-\pi+\pi^2-3),
  \]
  where $\pi$ is prime elements for the unique extension $v_L$ of
  $v_K$ to $L$. We see that 
  \[
       m := v_L(\pi-\sigma(\pi)) = 2.
  \]  
  The resulting singularity $\xi$ of the model $X$ of $X_K=\PP^1_K$
  constructed above is a rational triple point. 
\end{example}

\begin{remark}
  The generic fiber $X_K$ of our model $X$ is a curve of genus zero and so
  is not, strictly speaking, an example of the situation studied by
  Lorenzini. But we can easily modify our construction to get examples with
  arbitrary high genus. For instance, choose $m>1$, $p\nmid m$ and consider
  the Kummer cover $Y_K\to X_K$ of smooth projective curves with generic
  equation 
  \[
       Y_K: \; y^m =\phi(x).
  \]
  Then $g(Y_K)\geq 2$ (except for $p=3$ and $m=2$ when $g(Y_K)=1$). Let $Y$
  denote the normalization of the $R$-model $X$ inside the function field of
  $Y_K$. Then $Y$ is a (normal) $R$-model of $Y_K$. It can easily be shown
  that $Y$ has a unique singular point $\eta$ (which is the unique point in
  the inverse image of $\xi\in X$), and that $\eta\in Y$ is a wild quotient
  singularity in the sense of \cite{Lorenzini14}. We intend to study this
  situation in a subsequent paper.
\end{remark}       

\section{An explicit resolution}

To keep the construction of a desingularization in an explicit example as
concise as possible we now focus on the specific Example
\ref{exa:example1}. This case already illustrates the general situation quite
well, but is still sufficiently small to avoid lengthy explicit computations.

Set $K:=\QQ$ and let $v_K$ denote the $3$-adic
valuation on $K$ and $R:=\ZZ_{(3)}$ the valuation ring (the localizaton of
$\ZZ$ at $3$). Let $v_0$ denote the Gauss valuation on $K(x)$ with respect to
$x$. We define an inductive valuation $v$ on $K(x)$ as follows:
\[
      v := [v_0,\, v(x)=1/3,\, v(x^3-3x^2+3) = 2].
\]
Let $X$ be the model of $X_K:=\PP^1_K$ with $V(X)=\{v\}$. We have shown in the
preceeding section that $X$ has a unique singularity $\xi$ with a affine open
neighborhood $U=\Spec A$, where
\[
     A = R[x,y,z]/I,
\]
and where  $I$ is the ideal generated by the $2$-minors of the matrix
$$
   M= \begin{pmatrix} x &   y \cr
                      y &   z \cr
                      z & \;3x-3z-9
      \end{pmatrix}.
$$
The singular point $\xi$ corresponds to the maximal ideal $\m=(3,x,y,z)\lhd
A$.

\subsection{Explicit blowups and Tjurina modifications} 
\label{subsec:blowups} 

Our goal is to construct explicitly a desingularization $f:X'\to X$ of $\xi$.
For ease of notation we replace the projective scheme $X$ by the
affine open subset $U=\Spec A$. 

We not only know that $A$ is Cohen-Macaulay of codimension $2$, we are in an
even better setting, the situation of the Hilbert-Burch theorem, which then
implies that a free resolution of $A$ is of the form
$$0 \longrightarrow R[x,y,z]^2 \stackrel{M}{\longrightarrow} 
                    R[x,y,z]^3 \longrightarrow
                    R[x,y,z] \longrightarrow 
                    A \longrightarrow 0,$$
i.e. the Eagon-Northcott complex of $M$.\\

At first glance this seems to be unrelated to our task of desingularizing
$A$. However, these structural observations point us to well known results 
in the complex geometric case: In the late 1960s, Gergana
Tjurina classified the rational triple point singularities over the complex
numbers in \cite{Tj} and constructed minimal desingularizations thereof in 
a direct way. Our given matrix $M$ structurally corresponds to a singularity
of type $H_5$ in Tjurina's article, which we will refer to as $Y$ here and 
for which a presentation matrix (over ${\mathbb C}[x,y,z,w]$) is of the form
$$N=\begin{pmatrix} x & y \cr
                    y & z \cr
                    z & wx-w^2
  \end{pmatrix}.$$
The last entry can be replaced by $wx-wz-w^2$ without changing the
analytic type of the singularity as is shown in the classification of simple 
Cohen-Macaulay codimension $2$ singularities in \cite{FN}. This similarity 
suggests to try and mimic the philosophy of Tjurina's choice of centers for 
the desingularization of $X$.\\

Tjurina's first step towards a resolution of singularities is nowadays called
a Tjurina modification and is based on the observation that at each point 
of $Y$ except the origin the row space of the presentation matrix defines 
a unique direction in ${\mathbb C}^2$ and hence a point in the Grassmanian 
of lines in $2$-space. 
Resolving indeterminacies of this rational map into the Grassmanian then 
yields the Tjurina transform which can then be described by the equations 
$$N \cdot \begin{pmatrix} s \cr t \end{pmatrix}= 
          \begin{pmatrix} 0 \cr 0 \cr 0 \end{pmatrix}.$$
(For a more detailed treatment of Tjurina modifications see the first section
of \cite{FZ}.) Three further blow-ups, each at the ($0$-dimensional)
singular locus, which happens to be the non-normal crossing locus of the 
exceptional curves in the second and third blow-up, then lead Tjurina to a 
desingularization. The exceptional locus 
of this sequence of blow-ups consists of $6$ curves of genus zero, where 
the one originating from the Tjurina modification is the only one with 
self-intersection $-3$; all others have self-intersection $-2$. The dual
graph of the resolution is of the form:
\begin{figure}[h]
\begin{center}
 \input{graphfig1.tex}
 \caption{Tjurina's intersection graph $H_5$}
\end{center}
\end{figure}
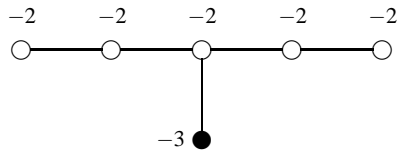

Returning to our setting, we can mimic these steps, obtaining the following
as ideal of the Tjurina transform:
$$I_{X_1}= \langle sx-ty,sy-tz,sz-t(3z-3x-9) \rangle $$
\label{tyurinamod}
By direct computation, it is easy to see that $X_1$ is regular except above 
$3$ and that above $3$ the non-regular locus is contained in the
chart $t \neq 0$. The exceptional curve $C_0$ which arose in this blow-up is a 
${\mathbb P}^1$ and corresponds to the ideal $\langle x,y,z,3 \rangle$. 
Passing to the chart $t \neq 0$, we can harmlessly eliminate
the variables $y$ and $z$ according to the first two generators. This 
essentially leaves a hypersurface described by the ideal
$$I_{X_1, {\rm new}}= \langle s^3x-3s^2x+3x+9 \rangle 
                    \subset R[x,s]$$
and an exceptional curve $I_{C_0}=\langle x,3 \rangle$. The non-regular locus 
of this hypersurface corresponds to $\langle x,s,3 \rangle$ as a direct 
computation shows; this is the center of the upcoming blow-up, which leads to
$3$ charts, two of which only contain regular points and only see normal 
crossing divisors. In the remaining chart ($y_1 \neq 0$), the strict transform 
is given by
$$I_{X_2}= \langle 3-y_2s, s^2y_0-s^2y_0y_2+y_0y_2+y_2^2 \rangle,$$
the strict transform of the exceptional curve $C_0$ by $\langle 3,y_0,y_2 \rangle$ and
the two components $C_1$ and $C_2$ of the new exceptional curve $E_2$ by 
$\langle 3,s,y_2(y_0+y_2) \rangle$. As the non-regular locus is given by
$\langle 3,s,y_0,y_2 \rangle$ and the non-normal crossing locus of the 
exceptional curves is the same point, analogous to Tjurina's setting, this 
point has to be chosen as upcoming center. 
After blowing up this point of $X_2$, we see in one chart that each of
the two components $C_1$ and $C_2$ of the preceding exceptional curve $E_2$ meets 
one component of the new exceptional curve $E_3$; more precisely, $C_1$ meets $C_3$
and $C_2$ meets $C_4$. In another chart, we see that the transform of $C_0$ meets 
both $C_3$ and $C_4$ at the origin, which is also the only singular point. 
Blowing up this point then introduces yet another exceptional curve $C_5$ 
meeting $C_0$, $C_3$ and $C_4$; at this stage, the strict 
transform is regular and the exceptional divisor is normal crossing. 
All exceptional curves are $-2$-curves except the $-3$-curve $C_0$.
Hence we obtained the dual graph:
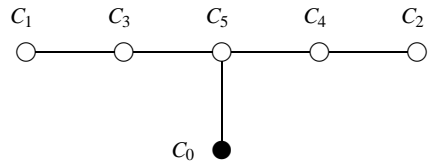
\begin{figure}[h]
\begin{center}
 \input{graphfig2.tex}
\caption{The intersection graph of the desingularization of $X$}\label{fig2}
\end{center}
\end{figure}

An explicit comparison of the computations of Tjurina and of the
one presented in our setting shows that that all computational steps as well
as the final result are analogous in both cases. This certainly raises the
question whether other singularities from Tjurina's list also have an analogue
arising from the construction of Section \ref{sec:wildquot} and what geometric
properties the singularities corresponding to the matrices of the previous
section might exhibit.

\begin{remark}
\begin{enumerate}
\item In the above calculation, we saw that we could safely replace the matrix $N$, 
which is the normal form in the classification of simple Cohen-Macaulay codimension $2$
singularities \cite{FN}, by a matrix say $N'$ which directly corresponds to the 
original matrix $M$, differing only by using a variable $w$ instead of $\pi_k=3$. 
The isomorphism of the local rings of the singularities represented by $N$ and 
$N'$ does not involve any change of $w$, whence we could hope for an equivalent 
isomorphism for $M$. This, however, does not exist, as the isomorphism over 
$\CC$ involves the multiplicative inverse of $3$.
\item As in the explicit example here, all the determinantal singularities from 
Proposition \ref{prop:example1}
allow a Tjurina modification at the origin of the respective chart at the
beginning of the desingularization; this provides an exceptional curve $C_0$. After 
this step, we see only one singular point, an $A_{pm-1}$ singularity. This latter 
singularity is well known to have a dual graph of resolution which is a chain with 
$pm-1$ vertices and $pm-2$ edges, where the middle vertex corresponds to the youngest 
exceptional curve. This middle vertex is the position, where the edge connecting the 
vertex corresponding to $C_0$ to the chain. 

\end{enumerate}
\end{remark}

\subsection{A posteriori description via
  valuations} \label{subsec:aposteriori} 

We return to our original notation, i.e.\ $X$ denotes the $R$-model of
$X_K=\PP^1_K$ with $V(X)=\{v\}$ (and not its affine subset $\Spec A$). Also,
$x$ again denotes the original coordinate function on $X_K$. 

The computation of the previous section show that there exist a
desingularization $f:X'\to X$ of $\xi$ such that the exceptional fiber
$E:=f^{-1}(\xi)$ is a normal crossing divisor and consists of $6$ smooth
rational curves, with an intersection graph given in Fig.\ \ref{fig2}.
The arithmetic surface $X'$ is itself an $R$-model of $X_K$ and is hence
completely determined by the set $V(X')$ of geometric valuations of $K(x)$
corresponding to the irreducible components of the special fiber $X_s'$. But
$X_s'$ consists precisely of the strict transform $C_6$ of $X_s$ (which
corresponds to the valuation $v_6:=v$) and the $6$ components $C_0,\ldots,C_5$
of the exceptional divisor.

The obvious question is: what are the valuations corresponding to the
components $C_i$, $i=0,\ldots,5$? 

\begin{proposition} \label{prop:aposteriori}
  Let $v_i$ denote the valuation on $K(x)$ corresponding to the component
  $C_i$, for $i=0,\ldots,5$. We normalize $v_i$ such that $v_i(3)=1$ (i.e.\
  such that $v_i|_K=v_K$). Then $v_0$ is the Gauss valuation with respect to
  the coordinate $x$. For $i=1,3,5$, 
  \[
     v_i  = [v_0,\,v_i(x) = r_i], \quad r_i = 
     \begin{cases}
        1/3, & i=5,\\
        1/2, & i=3,\\
        1,   & i=1.
     \end{cases}
  \]
  For $i=2,4$ we have
  \[
      v_i = [v_0,\,v_i(x)=1/3,\, v_i(\phi) = s_i],\quad s_i = 
     \begin{cases}
        4/3, & i=4,\\
        5/3, & i=2.
     \end{cases}
  \]
\end{proposition}

\begin{proof}
  This can be checked by a direct (but somewhat involved) computation, using
  the explict description of the desingularization by affine charts in \S
  \ref{subsec:blowups}. As an illustration of the general method let us
  convince ourselves that the Gauss valuation $v_0$ corresponds to the
  component $C_0$.

It suffices to consider the first step of the desingularization, the Tyurina
modification $X_1\to X$. We use the notation from p.\pageref{tyurinamod}. The
affine chart of $X_1$ defined by $t\neq 0$ has the form 
\[
    \Spec R[x_0,s\mid s^3x_0-3s^2x_0+3x_0+9=0]
\]
and the exceptional divisor $E_1\subset X_1$ is given on this chart by
$I_{E_1}=(x_0,3)$. So $\Spec \FF_3[s]$ is an affine open of $E_1$, and hence
$E_1$ is a projective line. We claim that $E_1$, as a prime divisor on $X$,
gives rise to the valuation $v_0$ (the Gauss valuation with respect to $x$). 

We write $x_0,s$ as rational functions in $x$:
\[
   x_0 = 9\phi^{-1}, \quad s = \frac{x_1}{x_0} = x.
\]
Now we see that the generators of the ideal $I_{E_1}$ have positive valuation
($v_0(3)=1$, $v_0(x_0)=2$) and $s$ is a $v_0$-unit and is a generator of its
residue field. This shows that the prime divisor $E_1\subset X_1$ corresponds
to the valuation $v_0$. As the component $C_0$ of the desingularizaton $X'\to X$
is simply the strict transform of $E_1$ under the map $X'\to X_1$, we have
proved the proposition for $i=0$. For $i=1,\ldots,5$ one can proceed in a
similar way. 
\hspace*{\fill}\qed
\end{proof}

\begin{remark}
\begin{enumerate}
\item We have found the set $V(X')=\{v_0,\ldots,v_6\}$ after computing the
  desingularization $X'\to X$. By Theorem \ref{thm:valuations}, $X'$ is
  determined by $V(X')$. Could we have found $V(X')$ by some other method, and
  would this give an alternative way to compute desingularization? In this
  simple case it is indeed possible to check the regularity of $X'$ (and the
  fact that $X_s'$ is a normal crossing divisor) purely in terms of the set of
  valuations $\{v_0,\ldots,v_6\}$. More details will be given elsewhere. 
\item
  If we accept that $X'$ is regular and $X_s'$ is a normal crossing divisor,
  it is easy to compute the self intersection numbers of the irreducible
  components $C_i$, as follows. Let 
  \[
      \tilde{E}:= (3) = \sum_{i=0}^6 m_i C_i \in{\rm Div}(X)
  \]
  be the principal divisor of the prime $3$. For each $i$ the integer $m_i$
  (the {\em multiplicity} of the component $C_i$) is equal to the ramification
  index of the extension $K(x)/K$ with respect to $v_i$. It is easy to read
  off $m_i$ from the explicit description of the $v_i$ in Proposition
  \ref{prop:aposteriori}:
  \[
      m_0=1,\; m_1=1,\; m_2=3,\; m_3=2,\; m_4 = 3,\; m_5 = 3,\ m_6= 3.
  \]
  Since $\tilde{E}$ is a principal divisor, we have
  \[
      0 = (C_i.\tilde{E}) = \sum_{j=0}^6 m_j (C_i.C_j),
  \]
  for $i=0,\ldots,6$, see e.g.\ \cite{SilvermanAT}, \S IV.7. The component
  graph from Fig.\ \ref{fig2} tells us what $(C_i.C_j)$ is for $i\neq
  j$ (either $1$ or $0$). Now the self intersection numbers $(C_i,C_i)$ can be
  computed easily. We find that
  \[
       (C_i.C_i) = 
        \begin{cases}
          -3, & i=0,\\
          -2, & i=1,\ldots,5,\\
          -1, & i=6.
        \end{cases}
  \]
\end{enumerate}
\end{remark}

\bibliographystyle{spmpsci} \bibliography{sources}

\end{document}

%% file: graphfig1.tex
\setlength{\unitlength}{1.2mm}
\begin{picture}(60,30)

\multiput(10,20)(10,0){5}{\circle{2}}
\put(30,10){\circle*{2}}

\put(11,20){\line(1,0){8}}
\put(21,20){\line(1,0){8}}
\put(31,20){\line(1,0){8}}
\put(41,20){\line(1,0){8}}
\put(30,19){\line(0,-1){8}}

\put(8.5,23){$-2$}
\put(18.5,23){$-2$}
\put(28.5,23){$-2$}
\put(38.5,23){$-2$}
\put(48.5,23){$-2$}
\put(25,9.3){$-3$}

\end{picture}


%% file: graphfig2.tex
\setlength{\unitlength}{1.3mm}
\begin{picture}(60,30)

\multiput(10,20)(10,0){5}{\circle{2}}
\put(30,10){\circle*{2}}

\put(11,20){\line(1,0){8}}
\put(21,20){\line(1,0){8}}
\put(31,20){\line(1,0){8}}
\put(41,20){\line(1,0){8}}
\put(30,19){\line(0,-1){8}}

\put(8.5,23){$C_1$}
\put(18.5,23){$C_3$}
\put(28.5,23){$C_5$}
\put(38.5,23){$C_4$}
\put(48.5,23){$C_2$}
\put(25,9.3){$C_0$}

\end{picture}
